\documentclass[12pt]{amsart}

\usepackage{amsmath,amsthm,amsfonts,latexsym,amssymb,amscd,color}
\usepackage{graphics,multimedia,tikz}
\numberwithin{equation}{section}
\theoremstyle{plain} %% This is the default
\newtheorem{thm}{Theorem}[section]

\newtheorem{lm}[thm]{Lemma}

\theoremstyle{definition}
\newtheorem{df}[thm]{Definition}

\newcommand{\wec}[1]{{\mathbf{#1}}}
\newcommand{\wrel}[1]{\;#1\;}  
\newcommand{\m}[1]{{\mathbf{\uppercase{#1}}}}

\newcommand{\Con}{{\mathrm{Con}}}

\newcommand{\Pol}{{\mathrm{Pol}}}
\newcommand{\Min}{{\mathrm{Min}}}
\newcommand{\tw}{{\mathrm{Tw}}}
\newcommand{\snag}[2]{{#1}^{[{#2}]}}

\begin{document}

%%%%%%%%%%%%%%%%%%%%%%%%%%%%%%%%%%%%%%%%%%%%%%%%%%%%%%%%%%%%%%%%%%%
%%  FRONT MATTER                                                 %%
%%%%%%%%%%%%%%%%%%%%%%%%%%%%%%%%%%%%%%%%%%%%%%%%%%%%%%%%%%%%%%%%%%%

\title[Supernilpotence]{Is supernilpotence super nilpotence?}

\author[K. A. Kearnes]{Keith A. Kearnes}
\email{kearnes@colorado.edu}
\address{Department of Mathematics\\
University of Colorado\\
Boulder, CO 80309-0395\\
USA}

\author[A. Szendrei]{\'Agnes Szendrei}
\email{szendrei@colorado.edu}
\address{Department of Mathematics\\
University of Colorado\\
Boulder, CO 80309-0395\\
USA}

\thanks{This material is based upon work supported by
  the National Science Foundation grant no.\ DMS 1500254,
  the Hungarian National Foundation for
  Scientific Research (OTKA) grant no.\ K115518, and the
  National Research, Development and Innovation Fund of
  Hungary (NKFI) grant no. K128042.}

\begin{abstract}
  We show that the answer
  to the question in the title
  is: ``Yes, for finite algebras.''
\end{abstract}

\maketitle

%%%%%%%%%%%%%%%%%%%%%%%%%%%%%%%%%%%%%%%%%%%%%%%%%%%%%%%%%%%%%%%%%%%
%%  MAIN MATTER                                                  %%
%%%%%%%%%%%%%%%%%%%%%%%%%%%%%%%%%%%%%%%%%%%%%%%%%%%%%%%%%%%%%%%%%%%

\section{Introduction}\label{intro}
The word ``supernilpotence'', with a specific meaning,
entered the general commutator theory lexicon
just over a decade ago
\cite{aichinger-mudrinski}.
The name suggests that some equation of the form
\begin{equation}\label{super_eq}
\textrm{supernilpotence = nilpotence + $\varepsilon$}
\end{equation}
should be true, but no such equation has ever been shown to hold
except in restricted settings.
Equation~(\ref{super_eq})
is meant to express that
supernilpotence implies nilpotence,
but that nilpotence does not always imply
supernilpotence.

In this paper we establish that Equation~(\ref{super_eq})
holds for finite algebras. The results were obtained in Fall 2017
and announced at the conference {\it Algebra and Lattices in Hawaii}
in Spring 2018.
The question of whether Equation~(\ref{super_eq})
holds for infinite algebras was posed at that conference,
and answered shortly afterwards by two groups of researchers.
The first solution came from
Matthew Moore and Andrew Moorhead,
who constructed in \cite{moore-moorhead}, for any $n>1$,
an algebra $\mathbb A_n$ that is $n$-step supernilpotent,
but not solvable of any degree, hence not nilpotent of any degree.
The second solution came from
Steven Weinell, who determined in \cite{weinell}
all possible higher commutator behaviors of simple algebras.
His work shows that there is
a simple algebra satisfying $[1,1,1]=0$ and $[1,1]=1$.
The first guarantees that the algebra is
$2$-step supernilpotent, while the second guarantees
that the algebra is neutral, which is stronger than
saying it is not solvable of any degree.

Let us describe the context of this research briefly.
The word ``nilpotent'' was introduced into mathematics
in \cite{peirce} to describe
an element $A$ of an associative algebra which satisfies
$\exists n\geq 2(A^n=0)$. The group-theoretic
concept of nilpotence was isolated in the paper
\cite{capelli}, which studied finite groups with
one Sylow $p$-subgroup for each $p$, i.e.
finite groups that factor as a product
of groups of prime power order.
In \cite{bruck}, the concept of
``central nilpotence'' of loops was studied.
This definition of nilpotence for loops agrees
with the commutator-theoretic definition
for groups, but does not agree with the
``prime power factorization into nilpotent factors''
definition from \cite{capelli}.
The difference in the definitions
was made clear in \cite{wright},
where it is shown that a finite loop $L$
has a prime power factorization into nilpotent factors
if and only if $L$ is centrally nilpotent
\emph{and} $L$ has a nilpotent multiplication group.
This result was extended in \cite{freespec}
to any variety of finite signature
which satisfies a congruence identity:
a finite algebra in such a variety
has a prime power factorization into nilpotent factors
if and only if it has a finite bound on the arity
of its nontrivial commutator terms if and only if
it is nilpotent in the sense of
ordinary commutator theory
\emph{and} has a twin monoid that is a nilpotent group.
The middle condition, having a finite bound on the arity
of nontrivial commutator terms, was shown to be
equivalent to supernilpotence for congruence permutable varieties
\cite{aichinger-mudrinski}, and later
this equivalence was extended to congruence modular varieties
in \cite{moorhead}.
Altogether, these results show that, for congruence modular varieties,
a finite algebra of finite signature is supernilpotent if and only if it
has a prime power factorization into nilpotent factors.

It is not difficult to show that these results extend
verbatim from congruence modular varieties
to varieties that omit type {\bf 1}. The reason for this
is that if $\m a$ is a finite supernilpotent algebra 
in a variety that omits type {\bf 1}, then the subvariety
generated by $\m a$ satisfies
$\textrm{typ}\{{\mathcal V}(\m a)\}\subseteq \{{\bf 2}\}$.
(To see why this is so, read the remark after Lemma~\ref{snags}.)
Therefore, if $\m a$ is a finite supernilpotent algebra
in a variety that omits type {\bf 1}, the subvariety
generated by $\m a$ 
will be congruence modular, and one can cite the results
that apply to congruence modular varieties.

But the presence of type ${\bf 1}$ in a variety
makes the problem difficult.
For example, it is easy to see that any finite algebra
that has a finite bound on the essential arities of its term
operations must be supernilpotent of type ${\bf 1}$, but already
in this special case it is not obvious that such algebras
must be nilpotent. (See \cite{strnil} for a proof
of nilpotence in this case.)

The purpose of this paper is to investigate the general case,
when type ${\bf 1}$ is present. We will first show that a
congruence on a finite algebra
that is supernilpotent
has local twin monoids that are nilpotent groups
(cf. \cite{sym, ham, strnil}).
Then we will argue that 
any congruence on a finite
algebra
that has local twin monoids that are nilpotent groups
must be nilpotent. This suffices to prove that
supernilpotence implies nilpotence for finite algebras.

\section{Supernilpotence for Finite Algebras}\label{finite}

Our goal is to prove that, for a congruence
$\beta$ of a finite algebra $\m a$, the
higher commutator condition 
\[
[\beta,\ldots,\beta]=0,\quad\textrm{with $k+1$\; $\beta$'s}
\]
implies the binary commutator conditions
\[
(\beta]^{\ell} = 0 = [\beta)^m
\]
for some $\ell$ and $m$. Here, for an arbitrary
congruence $\theta$, we define
$(\theta]^1=[\theta)^1 = \theta$,
$(\theta]^{\ell+1} = [\theta,(\theta]^{\ell}]$, and
$[\theta)^{m+1} = [[\theta)^{m},\theta]$.
We will connect the higher commutator to the
binary commutator through an intermediate property
involving the action of the 
$\beta$-twin monoids on minimal sets of $\m a$.
We recall the relevant
definitions and notation now.

\bigskip

Given congruences $\alpha, \beta\in\Con(\m a)$, 
the set $M(\alpha,\beta)$ of
{\bf $\alpha,\beta$-matrices}
is
\[
M(\alpha,\beta) =
\left\{\begin{bmatrix}
f(\wec{a},\wec{b})&f(\wec{a},\wec{b}')\\
f(\wec{a}',\wec{b})&f(\wec{a}',\wec{b}')
\end{bmatrix}\Bigg| \;f(\wec{x},\wec{y})\in\Pol(\m a),
\wec{a}\wrel{\alpha} \wec{a}',
\wec{c}\wrel{\beta} \wec{b}'
\right\}.
\]
The relation $C(\alpha,\beta;\delta)$ {\bf holds} if
\begin{equation} \label{implication}
  m_{11}\equiv m_{12}\pmod{\delta}\quad\quad\textrm{implies}\quad\quad
  m_{21}\equiv m_{22}\pmod{\delta}
\end{equation}
whenever
$\begin{bmatrix}
m_{11}&m_{12}\\
m_{21}&m_{22}
\end{bmatrix}\in M(\alpha,\beta)$.
One could indicate the two-dimensional nature
of an $\alpha,\beta$-matrix with
one of the simple diagrams

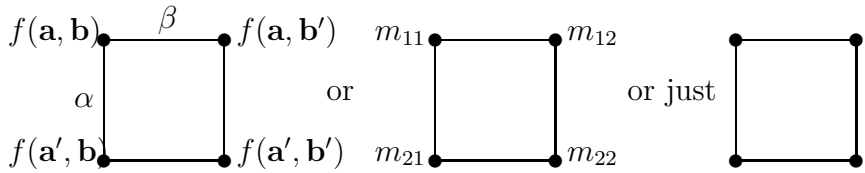
\begin{picture}(100,70)
\setlength{\unitlength}{.8mm}
\put(15,0){%
\begin{picture}(100,40)
\put(-5,0){\circle*{2}}
\put(15,0){\circle*{2}}
\put(-5,20){\circle*{2}}
\put(15,20){\circle*{2}}
\put(-5,0){\line(1,0){20}}
\put(-5,0){\line(0,1){20}}
\put(15,20){\line(-1,0){20}}
\put(15,20){\line(0,-1){20}}
\put(-21,20){$f(\wec{a},\wec{b})$}
\put(-21,0){$f(\wec{a}',\wec{b})$}
\put(17,20){$f(\wec{a},\wec{b}')$}
\put(17,0){$f(\wec{a}',\wec{b}')$}
\put(-10,9){$\alpha$}
\put(4,22){$\beta$}

\put(32,10){\textrm{or}}

\put(50,0){\circle*{2}}
\put(70,0){\circle*{2}}
\put(50,20){\circle*{2}}
\put(70,20){\circle*{2}}
\put(50,0){\line(1,0){20}}
\put(50,0){\line(0,1){20}}
\put(70,20){\line(-1,0){20}}
\put(70,20){\line(0,-1){20}}
\put(40,20){$m_{11}$}
\put(40,0){$m_{21}$}
\put(72,20){$m_{12}$}
\put(72,0){$m_{22}$}

\put(82,10){\textrm{or just}}

\put(100,0){\circle*{2}}
\put(120,0){\circle*{2}}
\put(100,20){\circle*{2}}
\put(120,20){\circle*{2}}
\put(100,0){\line(1,0){20}}
\put(100,0){\line(0,1){20}}
\put(120,20){\line(-1,0){20}}
\put(120,20){\line(0,-1){20}}
\end{picture}}
\end{picture}

\bigskip

\noindent
Implication (\ref{implication}) could be indicated by

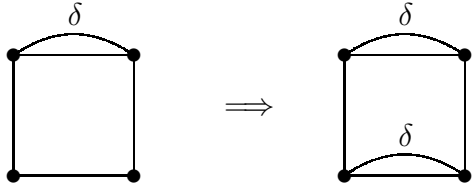
\begin{picture}(100,70)
\setlength{\unitlength}{.8mm}
\put(40,0){%
\begin{picture}(100,40)
\put(-5,0){\circle*{2}}
\put(15,0){\circle*{2}}
\put(-5,20){\circle*{2}}
\put(15,20){\circle*{2}}
\put(-5,0){\line(1,0){20}}
\put(-5,0){\line(0,1){20}}
\put(15,20){\line(-1,0){20}}
\put(15,20){\line(0,-1){20}}
\bezier{200}(-5,20)(5,27)(15,20)
\put(4,25){$\delta$}

\put(30,10){$\Longrightarrow$}

\put(50,0){\circle*{2}}
\put(70,0){\circle*{2}}
\put(50,20){\circle*{2}}
\put(70,20){\circle*{2}}
\put(50,0){\line(1,0){20}}
\put(50,0){\line(0,1){20}}
\put(70,20){\line(-1,0){20}}
\put(70,20){\line(0,-1){20}}
\bezier{200}(50,20)(60,27)(70,20)
\bezier{200}(50,0)(60,7)(70,0)
\put(59,5){$\delta$}
\put(59,25){$\delta$}
\end{picture}}
\end{picture}

\bigskip
\bigskip

\noindent
We define $[\alpha,\beta]$ to be the least $\delta$
such that $C(\alpha,\beta;\delta)$ holds.
We say that $\beta$ is abelian (or 1-step supernilpotent), and
we write $[\beta,\beta]=0$ for this,
exactly when $C(\beta,\beta;0)$ holds.

\bigskip

Now, using less detail than in the binary case,
given three congruences $\alpha, \beta,\gamma\in\Con(\m a)$, 
an {\bf $\alpha,\beta,\gamma$-matrix} is an object
that could be depicted

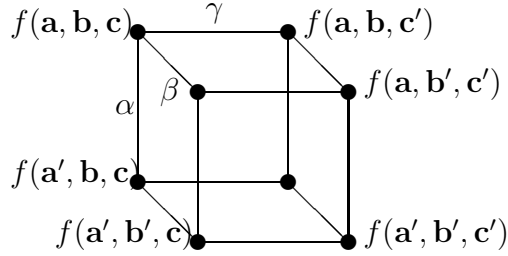
\begin{picture}(60,120)
\setlength{\unitlength}{1mm}
\put(45,10){%
\begin{picture}(50,40)
\put(-5,0){\circle*{2}}
\put(15,0){\circle*{2}}
\put(-5,20){\circle*{2}}
\put(15,20){\circle*{2}}
\put(-5,0){\line(1,0){20}}
\put(-5,0){\line(0,1){20}}
\put(15,20){\line(-1,0){20}}
\put(15,20){\line(0,-1){20}}

\put(-5,0){\line(1,-1){8}}
\put(15,20){\line(1,-1){8}}
\put(15,0){\line(1,-1){8}}
\put(-5,20){\line(1,-1){8}}

\put(3,-8){\circle*{2}}
\put(23,-8){\circle*{2}}
\put(3,12){\circle*{2}}
\put(23,12){\circle*{2}}
\put(3,-8){\line(1,0){20}}
\put(3,-8){\line(0,1){20}}
\put(23,12){\line(-1,0){20}}
\put(23,12){\line(0,-1){20}}

\put(-22,20){$f(\wec{a},\wec{b},\wec{c})$}
\put(-22,0){$f(\wec{a}',\wec{b},\wec{c})$}
\put(-16,-8){$f(\wec{a}',\wec{b}',\wec{c})$}
\put(17,20){$f(\wec{a},\wec{b},\wec{c}')$}
\put(25,12){$f(\wec{a},\wec{b}',\wec{c}')$}
\put(25,-8){$f(\wec{a}',\wec{b}',\wec{c}')$}

\put(-8,9){$\alpha$}
\put(-2,11){$\beta$}
\put(4,22){$\gamma$}
\end{picture}}
\end{picture}

\bigskip

\noindent
The relation 
$C(\alpha,\beta,\gamma;\delta)$ is defined by an implication
depicted

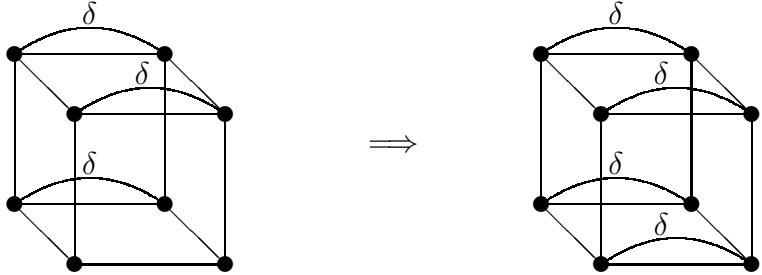
\begin{picture}(60,120)
\setlength{\unitlength}{1mm}
\put(10,10){%
\begin{picture}(50,40)
\put(-5,0){\circle*{2}}
\put(15,0){\circle*{2}}
\put(-5,20){\circle*{2}}
\put(15,20){\circle*{2}}
\put(-5,0){\line(1,0){20}}
\put(-5,0){\line(0,1){20}}
\put(15,20){\line(-1,0){20}}
\put(15,20){\line(0,-1){20}}

\put(-5,0){\line(1,-1){8}}
\put(15,20){\line(1,-1){8}}
\put(15,0){\line(1,-1){8}}
\put(-5,20){\line(1,-1){8}}

\put(3,-8){\circle*{2}}
\put(23,-8){\circle*{2}}
\put(3,12){\circle*{2}}
\put(23,12){\circle*{2}}
\put(3,-8){\line(1,0){20}}
\put(3,-8){\line(0,1){20}}
\put(23,12){\line(-1,0){20}}
\put(23,12){\line(0,-1){20}}

\bezier{200}(-5,0)(5,7)(15,0)
\put(4,4){$\delta$}
\bezier{200}(-5,20)(5,27)(15,20)
\put(4,24){$\delta$}
\bezier{200}(3,12)(13,19)(23,12)
\put(11,16){$\delta$}

\put(42,7){$\Longrightarrow$}

\put(65,0){\circle*{2}}
\put(85,0){\circle*{2}}
\put(65,20){\circle*{2}}
\put(85,20){\circle*{2}}
\put(65,0){\line(1,0){20}}
\put(65,0){\line(0,1){20}}
\put(85,20){\line(-1,0){20}}
\put(85,20){\line(0,-1){20}}

\put(65,0){\line(1,-1){8}}
\put(85,20){\line(1,-1){8}}
\put(85,0){\line(1,-1){8}}
\put(65,20){\line(1,-1){8}}

\put(73,-8){\circle*{2}}
\put(93,-8){\circle*{2}}
\put(73,12){\circle*{2}}
\put(93,12){\circle*{2}}
\put(73,-8){\line(1,0){20}}
\put(73,-8){\line(0,1){20}}
\put(93,12){\line(-1,0){20}}
\put(93,12){\line(0,-1){20}}

\bezier{200}(65,0)(75,7)(85,0)
\put(74,4){$\delta$}
\bezier{200}(65,20)(75,27)(85,20)
\put(74,24){$\delta$}
\bezier{200}(73,12)(83,19)(93,12)
\put(80,16){$\delta$}
\bezier{200}(73,-8)(83,-1)(93,-8)
\put(80,-4){$\delta$}

\end{picture}}
\end{picture}

\bigskip

(The image is asserting that
for every $\alpha,\beta,\gamma$-matrix, the implication holds.)

\bigskip

\noindent
We define $[\alpha,\beta,\gamma]$ to be the least $\delta$
such that $C(\alpha,\beta,\gamma;\delta)$ holds.
We say that $\beta$ is $2$-step supernilpotent, and write
$[\beta,\beta,\beta]=0$, exactly when $C(\beta,\beta,\beta;0)$ holds.

\bigskip

\noindent
In no detail at all, one can guess how the
set 
of
$\alpha_1,\ldots,\alpha_{k+1}$-matrices is defined,
and how the implication 
$C(\alpha_1,\ldots,\alpha_{k+1};\delta)$ is defined.
(Or, one can refer to
\cite[Definitions 2.1 and 2.8]{moorhead}
to see one way the notational complexities
in high dimensions might be handled. For this paper,
we only need sufficient understanding of the higher
commutator to understand the proof of Lemma~\ref{snags}.)

\bigskip

A congruence $\beta\in\Con(\m a)$
is {\bf $k$-supernilpotent}
or {\bf $k$-step supernilpotent}
if $C(\beta,\ldots,\beta;0)$
holds (with $k+1$ instances of $\beta$).
We also write $[\beta,\ldots,\beta]=0$
(with $k+1$ instances of $\beta$).

\bigskip

  If $\beta\in\Con(\m a)$, then a
{\bf $k$-dimensional $\beta$-snag}, or a
  {\bf $\snag{\beta}{k}$-snag},
  is a pair $(0,1)\in A^2$, $0\neq 1$,
  for which there is a
  $k$-dimensional matrix in $M(\beta,\ldots,\beta)$
  where $2^{k}-1$ entries have value $0$ and
  the remaining entry has value $1$.

  \bigskip

Recall from \cite[Definition~7.1]{hobby-mckenzie}
that, in tame congruence theory, a {\bf 2}-{\bf snag}
is a pair $(0,1)$, $0\neq 1$, for which there is a binary
polynomial $x\wedge y$ whose restriction to $\{0,1\}$
is the meet operation: $0\wedge 0=0\wedge 1=1\wedge 0=0$
and $1\wedge 1=1$. If $(0,1)\in \beta$ is a {\bf 2}-{snag},
then the matrix
\[
\left[\begin{array}{cc}
    0&0\\
    0&1
  \end{array}\right]
=
\left[\begin{array}{cc}
0\wedge 0&0\wedge 1\\
1\wedge 0&1\wedge 1
  \end{array}\right]\;\in\;M(\beta,\beta)
\]
witnesses that $(0,1)$ is a
$\snag{\beta}{2}$-snag.
Moreover, using the polynomial
$x_1\wedge x_2\wedge \cdots\wedge x_k$
(parenthesized in any way) in place of $x\wedge y$, one
can show that if $(0,1)\in \beta$ is a {\bf 2}-{snag},
then it is a
$\snag{\beta}{k}$-snag
for any $k$.

\begin{lm}\label{snags}
If
  \begin{itemize}
\item    
  $\m a$ is a finite algebra,
\item   $\beta\in \Con(\m a)$, and
\item  $\beta$ is $k$-supernilpotent, 
  \end{itemize}
then $\m a$ has no $\snag{\beta}{k+1}$-snags.
\end{lm}

\begin{proof}
  Any matrix witnessing that $(0,1)$ is a
  $\snag{\beta}{k+1}$-snag also witnesses that
  $C(\beta,\ldots,\beta;0)$ (with $k+1$ $\beta$'s) fails.
\end{proof}

Lemma~\ref{snags} is already strong enough to show
that supernilpotent congruences on finite algebras
are solvable, since (by the lemma
and the remarks preceding it) supernilpotence implies
the absence of {\bf 2}-snags,  
which implies
solvability according to \cite[Theorem~7.2]{hobby-mckenzie}.

A subset $U$ of an algebra $\m a$
is a {\bf neighborhood} if it is the image of some
idempotent unary polynomial.
(That is, $U=e(A)$ for some  $e(x)\in\Pol_1(\m a)$ satisfying $e(e(x))=e(x)$
on $\m a$.) Two unary polynomials $f(x), g(x)\in \Pol_1(\m a)$
are {\bf $\beta$-twins} if there is a polynomial
$h(x,\wec{y}) = h_{\wec{y}}(x)$ and $\beta$-related
parameter sequences $\wec{a}\;\beta\; \wec{b}$
such that $f(x) = h_{\wec{a}}(x)$ and 
$g(x) = h_{\wec{b}}(x)$ on $\m a$.
The {\bf $\beta$-twin monoid on $U$}, $\tw_{\beta}(\m a,U)$,
is the monoid of self-maps $f:U\to U$ 
induced by all polynomials $f$ of $\m a$ satisfying
\begin{enumerate}
\item $f(A)\subseteq U$, and
\item $f$ is a $\beta$-twin of some polynomial
  $g$ whose restriction to $U$ is the identity function on $U$.
\end{enumerate}

It is explained in 
\cite[Lemma~2.2]{freespec} why, when $\m a$ is finite,
there is a single polynomial $s_{\wec{y}}(x)$ and a single tuple $\wec{a}$
such that
$\tw_{\beta}(\m a,U)$ consists entirely of the functions
of the form $s_{\wec{b}}(x)|_U$ where $\wec{b}\;\beta\;\wec{a}$.
We call $s$ a {\bf generic polynomial} for
$\tw_{\beta}(\m a,U)$.

\begin{lm} \label{nilgrp}
If
  \begin{itemize}
\item    
  $\m a$ is a finite algebra,
\item   $\beta\in \Con(\m a)$, and
\item  $\m a$ has no $\snag{\beta}{k+1}$-snags,
\end{itemize}
then for any neighborhood $U$ of $\m a$,
  $\tw_{\beta}(\m a,U)$ acts on $U$ as a
nilpotent group of permutations, with nilpotence class
at most $k$.
\end{lm}

\begin{proof}
  The finite monoid $\tw_{\beta}(\m a,U)$ acts on $U$.
  If it contains a nonpermutation, then it
  contains an idempotent nonpermutation.
  If $\tw_{\beta}(\m a,U)$ contains
  an idempotent nonpermutation, then
  by iterating the generic polynomial
  $s_{\wec{y}}(x)$ as a function of $x$
  until it is idempotent in $x$ we obtain
  a polynomial
  $t_{\wec{y}}(x)$ for which there are
  $\beta$-related parameter sequences
  $\wec{b}, \wec{c}$ such that
  \begin{enumerate}
  \item[(i)]
    $t_{\wec{a}}(t_{\wec{a}}(x))=t_{\wec{a}}(x)$ and
    $t_{\wec{a}}(A)\subseteq U$ for all parameter sequences $\wec{a}$
  from $A$,
\item[(ii)] $t_{\wec{b}}(x)=x$ for $x\in U$, and
\item[(iii)]
  $t_{\wec{c}}(x)$
  is an idempotent nonpermutation of $U$.
  In particular, there are $u, v\in U$ such that
  $u\neq t_{\wec{c}}(u)=v=t_{\wec{c}}(v)$.
  \end{enumerate}
  Notice that $(u,v)=(t_{\wec{b}}(u),t_{\wec{c}}(u))\in\beta$,
  since $\wec{b}\wrel{\beta}\wec{c}$.

  Let
  $p(x,\wec{y}_1,\ldots,\wec{y}_k)=(t_{\wec{y_k}}\circ\cdots\circ t_{\wec{y_1}})(x)$.
  It is not hard to see that the $(k+1)$-dimensional
  matrix obtained from $p$ by making the $\beta$-related choices
  $x\in \{u,v\}$ and $\wec{y}_i\in\{\wec{b},\wec{c}\}$ is a 
  $\snag{\beta}{k+1}$-snag, since the value is $u$ if
  $x=u$ and $\wec{y}_i=\wec{b}$ for all $i$ and the value is $v$ otherwise.
  This shows that  if $\m a$ has no $\snag{\beta}{k+1}$-snags, then
$\tw_{\beta}(\m a,U)$ acts on $U$ as a group of permutations.

  Now suppose that
  $\tw_{\beta}(\m a,U)$
  acts on $U$ as a group of permutations,
  but not as a group of permutations of nilpotence
  class at most $k$.
  There must exist permutations in $\tw_{\beta}(\m a,U)$
  which fail to satisfy the group-theoretic
  identity
  \[
    [x_1,\ldots,x_{k+1}]=1,
    \]
  which asserts $k$-step nilpotence. More fully, this means that
  there exist permutations 
  $\gamma_1,\ldots, \gamma_{k+1}\in \tw_{\beta}(\m a,U)$
  such that the permutation of $U$ represented by the group commutator
  $[\gamma_1,\ldots, \gamma_{k+1}]$ is not the identity
  function on $U$. To unravel this statement even further,
  there exist $u\neq v$ in $U$ such that
  $[\gamma_1,\ldots, \gamma_{k+1}](v)=u$.
Since each $\gamma_i$ is a $\beta$-twin of the identity function on $U$,
  there must exist
  parameter sequences $\wec{b}, \wec{c}_i$, with $\wec{b}\;\beta\;\wec{c}_i$,
  such that
  for the generic polynomial $s_{\wec{y}}(x)$ we have 
  \begin{enumerate}
  \item[(i)]
$s_{\wec{b}}(x) = x$ on $U$
  \item[(ii)] $s_{\wec{c}_i}(x)=\gamma_i(x)$ on $U$.
  \end{enumerate}

$\m a$ has a polynomial $q(x,\wec{y}_1,\ldots,\wec{y}_{k+1})$
equal to 
$[s_{\wec{y}_1},\ldots, s_{\wec{y}_{k+1}}](x)$.
For this polynomial we have 
$q(x,\wec{c}_1,...,\wec{c}_{k+1})=[\gamma_1,...,\gamma_{k+1}](x)$,
which is a permutation of $U$ that maps $v$ to $u$.
But any other specialization of
$q(x,\wec{y}_1,...,\wec{y}_{k+1})$
with $\wec{y}_i\in\{\wec{b},\wec{c}_i\}$ results in a polynomial
which, as a function of $x$, is the identity permutation of $U$.
Consider the $(k+1)$-dimensional
matrix obtained from $q$ by fixing $x=v$ 
and making the $\beta$-related choices
$\wec{y}_i\in\{\wec{b},\wec{c}_i\}$.
This matrix will be a 
  $\snag{\beta}{k+1}$-snag, since the value is $u$ if
  each $\wec{y}_i$ is chosen to be $\wec{c}_i$, while
  the value is $v$ if any $\wec{y}_i$
  is chosen to be $\wec{b}$.
  In the contrapositive form,
  if $\m a$ has no   $\snag{\beta}{k+1}$-snags,
  then $\tw_{\beta}(\m a,U)$ acts on $U$ as a group of permutations
  of nilpotence class at most $k$.
\end{proof}

\begin{lm} \label{centrality}
If
  \begin{itemize}
\item    
  $\m a$ is a finite algebra,
\item   $\beta\in \Con(\m a)$,
\item $\delta\prec \theta$ in $\Con(\m a)$,
\item $U\in\Min_{\m a}(\delta,\theta)$, 
\item $N$ is a $\langle \delta,\theta\rangle$-trace of $U$, and
\item $\tw_{\beta}(\m a,U)$ acts as a group of permutations on $U$,
\end{itemize}
then $C(\beta,N^2;\delta)$ holds.
\end{lm}

\begin{proof}
  This is proved in 
  \cite[Lemma~3.4]{ham}
  in the case $\beta=1$, but the same argument
  holds for general $\beta$. Namely, the argument shows that
  if $C(\beta,N^2;\delta)$ fails, then
  $\tw_{\beta}(\m a,U)$ contains an idempotent
  nonpermutation.
\end{proof}

Now we recall from
\cite[Definition 4.12]{sym} the definition of a
$\beta$-regular quotient $\langle \delta, \theta\rangle$
of type~${\bf 1}$.

\begin{df}
  Let $\m a$ be a finite algebra with a tame quotient
$\langle \delta, \theta\rangle$  
  of type~${\bf 1}$, and let $\beta$ be a congruence of $\m a$.
  $\langle \delta, \theta\rangle$  is {\bf $\beta$-regular}
  if whenever
\begin{itemize}
\item
  $U\in\Min_{\m a}(\delta, \theta)$,
\item
  $N$ is a trace of $U$,
\item $H_{N,\beta}$
  is the subgroup of $\tw_{\beta}(\m a, U)$
  consisting of polynomials
  that map $N$ into itself,
\item   
  $p(x) \in H_{N,\beta}$ has a fixed point modulo $\delta$ on $N$
  (i.e., $p(u) \;\delta\; u$ for some $u\in N$),
\end{itemize}
then $p(x)$ is the identity modulo $\delta$ on $N$
(i.e., $p(x)\;\delta \;x$ for all $x\in N$).
\end{df}

The preceding definition hints that there
will be some group theory component to the
next part of the proof. We isolate
the fact from group theory that will be needed in our proof.

\begin{lm}\label{group}
If $G$ is finite group, $K$ is a
  core-free maximal subgroup of $G$, and
  $H\lhd G$ is a nilpotent normal subgroup,
  then $H$ is abelian and $H\cap K = \{1\}$.
\end{lm}

Recall that the {\bf core} of a subgroup $K$ of $G$ is 
the intersection of the $G$-conjugates of $K$. $K$
is {\bf core-free} in $G$ if its core is trivial,
equivalently $K$ is core-free if $\{1\}$ is the only subgroup
of $K$ that is normal in $G$.

\begin{proof}
  If $H\subseteq K$, then since $H\lhd G$ and $K$
  is core-free we get $H=\{1\}$, so $H$ is abelian and
  $H\cap K = \{1\}$.

  Assume $H\not\subseteq K$. Since $K\prec G$ and $H\lhd G$,
  we derive that   $HK=G$.
  Since the center of $H$ is characteristic
  in $H$, $Z(H)$ is normal in $G$.
  $Z(H)$ is nontrivial, since $H$ is nilpotent.
  $Z(H)\not\leq K$, since $K$ is core-free, so $G = Z(H)K$.
  Note that $Z(H)\cap K$ is normal in both $Z(H)$
  (since $Z(H)$ is abelian) and $K$ (since $Z(H)\lhd G$),
  hence $Z(H)\cap K\lhd Z(H)K=G$,
  hence $Z(H)\cap K$ is contained in the core of $K$.
  This shows that $Z(H)\cap K=\{1\}$.

  So far we have established that $Z(H)$ is a normal
  complement to $K$. This shows that any $g\in G$
  is uniquely representable as $g = z_gk_g$ where $z_g\in Z(H)$
  and $k_g\in K$. Moreover, if $h\in H$, then
  in the representation $h=z_hk_h$ we have $k_h = z_h^{-1}h\in H$,
  so in fact $k_h\in H\cap K$. This is enough to establish that  
  the function $\varphi:H\to Z(H): h\mapsto z_h$ is a
  group homomorphism. To see this, assume that
  $h=z_hk_h$ and $h'=z_{h'}k_{h'}$.
  Then 
  \[
hh'=z_hk_hz_{h'}k_{h'}=z_h(k_hz_{h'})k_{h'}
  \stackrel{?}{=}z_h(z_{h'}k_h)k_{h'}=(z_hz_{h'})(k_hk_{h'}).
  \]
  Here the equality $\stackrel{?}{=}$ is justified by the facts that
  $k_h\in H\cap K\subseteq H$ and $z_{h'}\in Z(H)$. Now the unique
  $Z(H)K$-representation of $hh'$ is both
  $z_{hh'}k_{hh'}$ and $(z_hz_{h'})(k_hk_{h'})$,
  so $z_{hh'}=z_hz_{h'}$, which is what it means for
  $\varphi$ to be a homomorphism.

  By examination one sees that $\varphi$ is the identity on
  its image $Z(H)$, and that the kernel of $\varphi$ is $H\cap K$.
  Therefore $H\cap K$ is a normal complement to $Z(H)$ in $H$.
  Thus forces $H\cong Z(H)\times (H\cap K)\cong Z(H)\times H/Z(H)$.
  The rightmost side
  has smaller nilpotence degree than the leftmost unless $H=Z(H)$
  and $H\cap K=\{1\}$, so we conclude that
  $H$ is abelian and $H\cap K=\{1\}$,
  as desired.
\end{proof}

\begin{lm} \label{regular}
If
  \begin{itemize}
\item    
  $\m a$ is a finite algebra,
\item   $\beta\in \Con(\m a)$,
\item $\delta\prec \theta$ is a type-{\bf 1} covering in $\Con(\m a)$,
\item $U\in\Min_{\m a}(\delta,\theta)$, 
\item $N$ is a $\langle \delta,\theta\rangle$-trace of $U$, and
\item $\tw_{\beta}(\m a,U)$ acts as a nilpotent
  group of permutations on $U$,     
\end{itemize}
then   $\langle \delta, \theta\rangle$  is $\beta$-regular.
\end{lm}

\begin{proof}
  Let $\m m = \m A|_N/\delta|_N$.
  By 
  \cite[Corollary~5.2~(1)]{hobby-mckenzie}, $\m m$ is a minimal
  algebra of type ${\bf 1}$.   By the definition of type ${\bf 1}$
  (\cite[Definition~4.10]{hobby-mckenzie}), $\m m$ is polynomially
  equivalent to a $G$-set.
By  \cite[Lemma 2.4]{hobby-mckenzie} and the fact that
  $\delta\prec\theta$, $\m m$ is simple.

  Let $G$ be the group of polynomial permutations of $\m m$.
  Let $H = H_{N,\beta}/\delta$ be the subgroup of $G$
  represented by $\beta$-twins of the identity.
  Note that $H\lhd G$, since any conjugate of a twin of the identity
  is a twin of a conjugate of the identity, hence
  is a twin of the identity. Also note that $H$ is nilpotent,
  since it is a quotient of $H_{N,\beta}$, which is a subgroup
  of the nilpotent group $\tw_{\beta}(\m a,U)$.
  Finally let $K=K_{u/\delta}$
  be the stabilizer of some point $u/\delta\in N/\delta=M$.
  To prove that
  $\langle \delta, \theta\rangle$  is $\beta$-regular we must
  show that any $\beta$-twin of the identity in $G$
  (i.e., an element
  representing an element of $H$) which has some fixed point
  $u/\delta$ on $N/\delta=M$
  (i.e., which lies in some $K=K_{u/\delta}$)
  must be the identity modulo $\delta$
  (i.e., represents the identity element of $G$).
  In short, we want to prove that $H\cap K=\{1\}$.

  As a first case, assume that $\m m$ is a 
  discrete
  $G$-set (so $G$, $H$, $K$ are all trivial). Clearly
  $H\cap K=\{1\}$.

  The remaining case is the one where $\m m$ is not
  discrete. Since $\m m$ is simple, $G$ acts primitively
  on $M$, and
  any 1-point stabilizer $K=K_{u/\delta}$ is a maximal
  subgroup of $G$.
  We are now in the situation of Lemma~\ref{group}, so
  $H\cap K = \{1\}$, as desired.
\end{proof}

\begin{thm}
  Suppose that $\m a$ is a finite algebra,
  and $\beta\in \Con(\m a)$.
 The following implications
  hold among the listed properties:
$(1)\Rightarrow(2)\Rightarrow(3)\Rightarrow(4)\Rightarrow(5)$.
  \begin{enumerate}
  \item[(1)]
    $\beta$ is $k$-supernilpotent.
  \item[(2)]
$\m a$
  has no $\snag{\beta}{k+1}$-snags.
  \item[(3)]
$\tw_{\beta}(\m a,U)$ acts as a nilpotent
    group of permutations,  with nilpotence class at most $k$,
    on any minimal set $U$ of $\m a$.
\item[(4)]
  $C(\beta,\theta;\delta)$ and 
  $C(\theta,\beta;\delta)$ hold whenever $\delta\prec\theta$
  in $\Con(\m a)$.
\item[(5)]
  $\beta$ is (both left and right) nilpotent.
    \end{enumerate}
\end{thm}

\begin{proof}

The fact that $(1)\Rightarrow(2)$
is proved in Lemma~\ref{snags}.

The fact that $(2)\Rightarrow(3)$
is a consequence of Lemma~\ref{nilgrp}.

Now we prove that $(3)\Rightarrow(4)$.
Assume (3). We first prove that $\beta$ is left
  nilpotent, and then we prove (4).

Assume the sequence $\beta = (\beta]^1\geq (\beta]^2\geq \cdots$
stabilizes at $\theta$, that is that $\theta = \bigcap_i (\beta]^i$.
If $\theta = 0$, then $\beta$ is left nilpotent and there
is nothing more to prove at this point.
Otherwise $\theta=(\beta]^N=(\beta]^{N+1}>0$ for some $N$.
Choose $\delta$ so that $\delta\prec\theta$
in this case.

We cannot have $C(\beta,\theta;\delta)$, else
$[\beta,\theta]=(\beta]^{N+1}\leq \delta < \theta=(\beta]^N$,
which contradicts the choice of $\theta$.
But we do have $C(\beta,N^2;\delta)$ for some (any)
$\langle \delta,\theta\rangle$-trace,
according to Lemma~\ref{centrality}.
Using terminology from
\cite{sym}, the fact that $\beta$ centralizes
a $\langle \delta,\theta\rangle$-trace $N$ but does not centralize
the entire congruence quotient means that 
$\langle \delta,\theta\rangle$ is not $\beta$-coherent.
From \cite[Lemma~4.2]{sym},
the type of
$\langle \delta,\theta\rangle$ must be ${\bf 1}$.
From 
\cite[Lemma~4.13]{sym},
$\langle \delta,\theta\rangle$ cannot be $\beta$-regular.
However, we proved in Lemma~\ref{regular} that
$\langle \delta,\theta\rangle$ is $\beta$-regular
when it is of type ${\bf 1}$.
This contradicts our assumption that
$\beta = (\beta]^1\geq (\beta]^2\geq \cdots$
    stabilizes at some $\theta>0$, so we may conclude
    that $\beta$ is left nilpotent.

Now we change notation and allow
$\langle \delta,\theta\rangle$ to be
an arbitrary prime quotient of $\m a$.
It is shown in 
\cite[Lemmas~3.1 and 3.2]{sym}
that the following conditions are equivalent
when $\beta\in\Con(\m a)$ is left nilpotent
and the type of the prime quotient $\langle \delta,\theta\rangle$
is not ${\bf 1}$:
\begin{enumerate}
\item[(i)] $C(\beta,\theta;\delta)$,
\item[(ii)] $[\beta,\theta]\leq\delta$,
\item[(iii)] $C(\theta,\beta;\delta)$, and
\item[(iv)] $[\theta,\beta]\leq\delta$.
\end{enumerate}
Moreover, we can add the following
equivalent condition to this list
when the type of $\langle \delta,\theta\rangle$
is not ${\bf 1}$:
\begin{enumerate}
\item[(v)] $C(\beta,N^2;\delta)$ for some (any)
$\langle \delta,\theta\rangle$-trace $N$.  
\end{enumerate}
The reason that this can be added is that
$C(\beta,\theta;\delta)\Rightarrow C(\beta,N^2;\delta)$
always holds,
since $N^2\subseteq \theta$, while 
\cite[Lemma~4.2]{sym} proves that the reverse implication
$C(\beta,N^2;\delta)\Rightarrow C(\beta,\theta;\delta)$
can only fail when the type of
$\langle \delta,\theta\rangle$ is ${\bf 1}$.
We proved in Lemma~\ref{centrality} that (v)
holds under
our assumption (3), so we derive here that 
each of (i)--(iv) from above hold when the type
of $\langle \delta,\theta\rangle$ is not ${\bf 1}$.

When the type
of $\langle \delta,\theta\rangle$ \emph{is} ${\bf 1}$,
we can argue the same way as long as
$\langle \delta,\theta\rangle$ is $\beta$-regular,
according to 
\cite[Lemma~4.14]{sym}.
Since we established $\beta$-regularity in Lemma~\ref{regular},
we are done.

Now it is easy to show that $(4)\Rightarrow(5)$.
From (4) it follows that if $\theta$ is any
nonzero congruence of $\m a$, then
$[\beta,\theta]<\theta$ and 
$[\theta,\beta]<\theta$, so by the finiteness of $\Con(\m a)$
any mixed commutator expression
involving enough $\beta$'s must equal zero.
  \end{proof}

\bibliographystyle{plain}

\end{document}